\documentclass[11pt]{amsart}
\usepackage{amsmath,amsthm,amssymb, amsfonts,color}
\usepackage{enumerate}
\usepackage[all]{xy}
\usepackage{graphicx}
\usepackage{fullpage}
\usepackage{hyperref}
\usepackage{mathrsfs}
\usepackage{comment}
\usepackage{mathtools}

\theoremstyle{plain}
\newtheorem{thm}{Theorem}[section]
\newtheorem{lem}[thm]{Lemma}
\newtheorem{prop}[thm]{Proposition}

\theoremstyle{definition}
\newtheorem{definition}[thm]{Definition}
\newtheorem{remark}[thm]{Remark}
\newtheorem*{remark*}{Remark}

\newtheorem*{thm*}{Theorem}
\newtheorem*{lem*}{Lemma}

\newcommand{\Spec}{\operatorname{Spec}}
\newcommand{\Proj}{\operatorname{Proj}}

\newcommand{\Ann}{\operatorname{Ann}}

\newcommand{\Cl}{\operatorname{Cl}}

\newcommand{\Span}{\operatorname{Span}}
\newcommand{\Sym}{\operatorname{Sym}}

\newcommand{\Br}{\operatorname{Br}}
\newcommand{\Coh}{\operatorname{Coh}}

\newcommand{\Cohproj}{\operatorname{Cohproj}}
\newcommand{\cone}{\operatorname{cone}}
\newcommand{\adj}{\operatorname{adj}}
\newcommand{\Hilb}{\operatorname{Hilb}}

\def\C {{\mathbb C}}
\def\P {{\mathbb P}}

\begin{document}
\title{Derived Partners of Enriques Surfaces}
\author{Lev Borisov, Vernon Chan, Chengxi Wang}
\maketitle

\textbf{Abstract.} Let $V$ be a $6$-dimensional complex vector space with an involution $\sigma$ of trace $0$, and let $W \subset \Sym^2 V^\vee$ be a generic $3$-dimensional subspace of $\sigma$-invariant quadratic forms. To these data we can associate an Enriques surface as the $\sigma$-quotient of the complete intersection of the quadratic forms in $W$. We exhibit noncommutative Deligne-Mumford stacks together with sheaves of Azumaya algebras on them whose derived categories are equivalent to those of the Enriques surfaces. This provides a more accessible treatment of \cite[Theorem 6.16]{kp}. We also construct geometric realizations of the Brauer classes coming from these sheaves of Azumaya algebras which may be of independent interest.

\section{Introduction}

Derived categories are a fascinating invariant of algebraic varieties. Two algebraic varieties are called derived equivalent if their bounded derived categories of coherent sheaves are equivalent as triangulated categories \cite{huybrechts}. Bondal and Orlov  \cite{orlov} showed that a smooth projective variety $X$ whose canonical class  $K_X$ is ample or anti-ample is determined by its bounded derived category $\mathcal{D}^b(\Coh-X)$, so in these cases derived equivalence implies isomorphism. However in the Calabi-Yau case of $K_X = 0$ there are multiple constructions of non-trivial derived equivalences \cite{caldararu}. In a number of  cases, derived equivalences are related to string theory: for Calabi-Yau varieties $X, Y$ and $Z$, homological mirror symmetry predicts that if $X$ and $Y$ are both ``mirror'' to $Z$ (we say that $X$ and $Y$ are a double mirror pair), then $X$ and $Y$ are derived equivalent.

\medskip
Many examples of derived equivalences come from the Homological Projective Duality construction of Kuznetsov \cite{kuznetsov2}. In these examples, the derived partners of a Calabi-Yau variety can be slightly non-commutative and involve a DM-stack structure or a Brauer class. Of particular interest to us is the example of complete intersections of quadrics \cite{kuznetsov}, especially in dimension 2, which we describe briefly here. For $V \cong \mathbb{C}^6$ and $W \subset S^2V^\vee$ a  dimension $3$ subspace of homogeneous quadratic forms on $V$ such that the intersection $X$ of quadrics parametrized by $W$ is a complete intersection, there exists a sheaf of algebras $A$ on a double cover $Y$ of $\mathbb{P}W$ ramified over the locus of singular quadrics, such that the derived category  $(Y,A)$ is equivalent to the derived category of $X$:
\begin{equation}\label{dual-Cliff}
\mathcal{D}^b(\Coh-Y,A) \quad \simeq \quad \mathcal{D}^b(\Coh-X).
\end{equation}
In \cite{borisov}, Borisov and Li proposed a toric framework to generalize the above construction. In particular, they argued for the existence of a derived partner of Enriques surfaces (see \cite[Section 9.2]{borisov}), making use of the fact that generic complex Enriques surfaces can be obtained as a quotient of $(2,2,2)$-complete intersections in $\mathbb{P}^5$ by a fixed-point-free involution \cite{cossec}. In  this paper, we look into the case of Enriques surfaces in more detail. 

\medskip
As above, let $V$ be a $6$-dimensional complex vector space. However, it is now equipped with an involution $\sigma$ of trace zero. Let  $X$ be a K3 surface obtained as a $(2,2,2)$-complete intersection in $\mathbb PV\cong\mathbb{CP}^5$, with the defining space of quadrics $W$ being a generic dimension $3$ subspace of  the eigenvalue one subspace of $S^2V^\vee$. Then $\sigma$ induces a fixed-point-free involution on $X$ such that $X/\sigma$ is an Enriques surface.\footnote{ A generic Enriques surface can be constructed using this process \cite{cossec}.} On the dual side, the locus of singular quadrics in $\P W\cong\C\P^2$ is the union of two smooth cubic curves which intersect transversely. The corresponding double cover $Y\to \mathbb{P}^2$ has $9$ nodal singularities. We construct a smooth stacky resolution $\mathscr{Y}$ of $Y$. We are also interested in the quotient $\widehat{\mathscr{Y}}$  of $\mathscr{Y}$ by a lift of the covering involution of $Y\to \P W$, which is isomorphic to the $(2,2)$-root stack of $\P W$ along the aforementioned cubic curves.

\medskip
Our main results are the following.

\medskip\noindent
{\bf Theorem \ref{main_thm_1}}
There exists a sheaf of Azumaya algebras $\mathcal A$ on  $\mathscr{Y}$ so that there is a derived equivalence
\[
\mathcal{D}^b(\Coh-\mathscr{Y},\,\mathcal A) \quad \simeq \quad \mathcal{D}^b(\Coh-X).
\]

\medskip\noindent
{\bf Theorem \ref{main_thm_2}}
There exists a sheaf of Azumaya algebras $\widehat{\mathcal A}$ on  $\widehat{\mathscr Y}$ so that there is a derived equivalence
\[
\mathcal{D}^b(\Coh-\widehat{\mathscr Y},\,\widehat{\mathcal A}) \quad \simeq \quad \mathcal{D}^b(\Coh-X/\sigma).
\]

Here a sheaf of Azumaya algebras over a stack $\mathscr{X}$ is a coherent sheaf of associative algebras which is \'{e}tale locally isomorphic to a matrix algebra sheaf.

\medskip
In addition, we construct the geometric realization of the Brauer classes on $\mathscr{Y}$ and $\widehat{\mathscr Y}$ that correspond to $\mathcal A$ and $\widehat{\mathcal A}$ as Severi-Brauer varieties over the stacks $\mathscr{Y}$ and $\widehat{\mathscr Y}$.

\medskip
\textit{Structure of the paper.} In Section 2 we define the notations and list some properties of the quotient stacks involved. In Section 3 we construct the Azumaya algebras in question in terms of Clifford algebras and provide the proofs of the main theorems. In Section 4 we construct the associated Severi-Brauer varieties as geometric realizations of the Brauer classes. In Section 5 we discuss further research directions and open problems.

\medskip
\textit{Acknowledgements.} We thank Alexander Kuznetsov, Daniel Krashen and Franco Rota for insightful suggestions and discussions. We thank the anonymous referees for suggestions on improving the exposition and for catching minor mistakes in the first version of the paper. We note that the main results of this paper have also appeared in \cite{kp}, although our setup is more explicit. 
We do not know whether our equivalence of categories is the same as that of \cite{kp}.

\section{Quotient stacks associated to the K3 and Enriques surfaces}

In this section we define the schemes and stacks  $X$, $X/\sigma$, $\mathscr{Y}$ and $\widehat{\mathscr Y}$ that we will be using in the paper, and fix our notations in the process.

\medskip
Let $V$ be a $6$-dimensional complex vector space with coordinates $x_1^\pm, x_2^\pm, x_3^\pm\in V^\vee$ equipped with an involution $\sigma$ that fixes $x_i^+$ and sends $x_i^- \mapsto -x_i^-$. Let $W \subset S^2V^\vee$ be a general dimension $3$ subspace of homogeneous quadratic forms which are invariant under the involution. Each quadratic form $q$ in $W$ defines a quadric in $V$. Let $X \subset \mathbb{P}V$ be the zero locus of the quadric forms in $W$.  
We fix a basis $(q_1,q_2,q_3)$ of $W$ and consider  the K3 surface $X$ which is the complete intersection of $3$ quadrics $\{0=q_i\}, 1\leq i\leq 3$. If we take the quotient of $X$ by the involution $\sigma$ on $V$,  we obtain an Enriques surface $X/\sigma$. 

\medskip
On the double mirror side,\footnote{There is indeed a common mirror family, but we will not be using it.} let $u_1, u_2, u_3$ be coordinates on $W \cong \mathbb{C}^3$ so that every quadratic equation $q$ in $W$ can be written as $q=\sum_i u_i q_i$. As each $q \in W$ is invariant under the involution, $q$ can be written as $q = q^+ + q^-$ where each of $q^\pm$ involves $x_i^\pm$ only, and the matrix representation of $q$ is a block matrix consisting of two $3 \times 3$ matrices representing $q^+$ and $q^-$. Thus, $\det(q) = 0$ is a sextic equation on $W$ in $u_i$ and the curve $E$ it defines in $\mathbb{P}W$ is the union of two cubic curves $E_+$ and $E_-$ respectively defined by
\[
f_\pm \; := \; \det(q^\pm).
\]
As $W$ is assumed to be a general subspace, and the loci for which the rank of $q_+$ or $q_-$ drops to $1$ are cut out by the $2 \times 2$ minors of their associated matrices, we see that all nonzero quadratic forms in $W$ have rank $\geq 4$. 

\medskip
In the original setting, Kuznetsov considered the double cover $Y$ of $\mathbb{P}W$ ramified over $E$, which is generically a smooth sextic curve. In our case, the double cover would be singular because the ramification locus $E_+ \cup E_-$ has nodes, so we want to look at crepant desingularizations of it. There are at least $2$ options: to blow up the nodes which gives nine $(-2)$-curves; or to consider stack structures on the nodes. In this paper, we go with the second route, but the other option also appears interesting and may be related to construction of Enriques surfaces via logarithmic transformation of rational elliptic surfaces.

\medskip
Specifically, let us define the stacky resolution $\mathscr{Y}$ of the ramified double cover $Y$ as a global quotient stack. We first define the following algebras over the coordinate ring $\mathbb{C}[u] := \mathbb{C}[u_1, u_2, u_3]$ of $W$.
\begin{align*}
    A &:= \mathbb{C}[u_1, u_2, u_3, y_+, y_-] / \langle y_+^2-f_+,\, y_-^2-f_- \rangle, \\
    B &:= \mathbb{C}[u_1, u_2, u_3, y] / \langle y^2-f_+f_- \rangle.
\end{align*}
These algebras are related by the injective algebra homomorphism $B \to A$ which sends $u_i\mapsto u_i$ and  $y \mapsto y_+ y_-$. Also note that both $A$ and $B$ are free $\C[u]$ modules, with the generating sets $\{1,y_+,y_-,y_+y_-\}$ and $\{1,y\}$ respectively. Next, we define group actions of $\mathbb{C}_\lambda^*$, $\mathbb{C}_t^* \cong \mathbb{C}^*$ on $A$ and $B$ respectively:
\begin{align*}
    \lambda \cdot (u_1, u_2, u_3, y_+, y_-) &:= (\lambda^2 u_1, \lambda^2 u_2, \lambda^2 u_3, \lambda^3 y_+, \lambda^3 y_-), \\
    t \cdot (u_1, u_2, u_3, y) &:= (t u_1, t u_2, t u_3, t^3 y).
\end{align*}
Here we added subscripts $\lambda$ and $t$ under the groups to distinguish two different $\mathbb{C}^*$-actions. The two actions are related by the surjective group homomorphism  $\mathbb{C}_\lambda^* \to \mathbb{C}_t^*$ where $\lambda \mapsto t = \lambda^2$, which is compatible with the map $B \to A$. Passing to the quotient stacks, we obtain a map
\[
[ (\Spec A \backslash 0) / \mathbb{C}_\lambda^* ] \to [ (\Spec B \backslash 0) / \mathbb{C}_t^* ].
\]
Here and for the remainder of the paper, we abbreviate
\[
\Spec A \backslash 0 := (\Spec A) \,\backslash\, \{(u_1, u_2, u_3) = 0\} = (\Spec A) \,\backslash\, \{(u_1, u_2, u_3, y_+, y_-) = 0\},
\]
and do the same to $W\backslash 0$ etc. The $u_i$ coordinates will often be collectively referred to as $u$.

\medskip
The $\mathbb{C}_t^*$-action on the $u_i$ coordinates is just the scaling action on the projective space $\mathbb{P}W$. There are $\mathbb{N}$-gradings on $A$ and $B$ corresponding to the two $\mathbb{C}^*$-actions, explicitly:
\[
A: 
\begin{cases}
u_1, u_2, u_3: & \text{degree } 2 \\
y_+, y_-: & \text{degree } 3
\end{cases}
\quad \text{and} \quad
B: 
\begin{cases}
u_1, u_2, u_3: & \text{degree } 1 \\
y: & \text{degree } 3
\end{cases}.
\]
Then the ramified double cover $Y \to \mathbb{P}W$ can be realized as $\Proj B \to \mathbb{P}W$.

\begin{remark}
\label{rmk_grading}
The algebra $A$ cannot be made into a graded algebra over $\mathbb{C}[u]$ that respects the natural grading of $u_i$'s in $\mathbb{P}W$ (i.e. $u_i$'s have degree $1$), because then $y_\pm$ would have to have half-integer weight.
\end{remark}

We now explore the smoothness conditions on $\Spec A \backslash 0$.
\begin{prop}
\label{spec_a_smooth} The variety
$\Spec A \backslash 0$ is smooth if and only if $E_+$, $E_-$ are smooth and intersect transversely. In particular, it is smooth for generic choices of $W$.
\end{prop}

\begin{proof}
The curves $E_+$, $E_-$ are smooth if and only if their Jacobians $[\partial f_+/\partial x_i^{\pm}]$ and $[\partial f_-/\partial x_i^{\pm}]$ are non-vanishing on themselves, and $E_+$ intersects $E_-$ transversely means that the $2 \times 6$ matrix
\[
\begin{bmatrix}
\frac{\partial f_+}{\partial x_1^+} & \frac{\partial f_+}{\partial x_2^+} & \dots & \frac{\partial f_+}{\partial x_3^-} \\
\frac{\partial f_-}{\partial x_1^+} & \frac{\partial f_-}{\partial x_2^+} & \dots & \frac{\partial f_-}{\partial x_3^-}
\end{bmatrix}
\]
has rank $2$ at the points of $E_+ \cap E_-$. The Jacobian for $\langle y_+^2-f_+,\, y_-^2-f_- \rangle$ is the $2 \times 8$ matrix
\[
\left[
\begin{array}{cccc|cc}
-\frac{\partial f_+}{\partial x_1^+} & -\frac{\partial f_+}{\partial x_2^+} & \dots & -\frac{\partial f_+}{\partial x_3^-} & 2y_+ & 0 \\
-\frac{\partial f_-}{\partial x_1^+} & -\frac{\partial f_-}{\partial x_2^+} & \dots & -\frac{\partial f_-}{\partial x_3^-} & 0 & 2y_-
\end{array}
\right]
\]
so it has rank $2$ at the points of $\mathbb{P}W$ outside of $E_+ \cup E_-$ because of the second block matrix. Looking at the first block matrix, we see that it has rank $2$ at the points of $E_+ \cup E_-$ precisely when $E_+$, $E_-$ are smooth and intersect transversely.
\end{proof}

\begin{prop}
\label{stabilizer_resolution}
The $\mathbb{C}_\lambda^*$-action on $\Spec A\backslash 0$ has stabilizer groups at a point $(u,y_\pm)$:
\[
\begin{cases}
    \text{trivial} \quad &\text{if} \quad y_+ \neq 0 \text{ or } y_- \neq 0, \\
    \mathbb{Z}_2 \quad &\text{if} \quad y_+,\, y_- = 0.
\end{cases}
\]
\end{prop}

\begin{proof}
To compute the stabilizer, we require that
\[
(\lambda^2 u_1, \lambda^2 u_2, \lambda^2 u_3, \lambda^3 y_+, \lambda^3 y_-) = (u_1, u_2, u_3, y_+, y_-),
\]
so $\lambda$ must be $\pm 1$ by looking at the $u_i$'s. Then $\lambda = -1$ fixes $(u, y_\pm)$ iff $y_+ = y_- = 0$.
\end{proof}

\begin{definition}
\label{def_resolution}
We denote by $\mathscr{Y}$  the quotient stack $[\, (\Spec A \backslash 0) \,/\, \mathbb{C}_\lambda^*\, ]$.
\end{definition}


\begin{prop}
\label{prop_resolution} The Deligne-Mumford stack
$\mathscr{Y}$ is a stacky resolution of $Y$, and $Y$ is its coarse moduli space. Furthermore, $\mathscr{Y}$ has $\mathbb{Z}_2$-stack structure at the $9$ intersection points of $E_+$ and $E_-$, and ordinary scheme points elsewhere.
\end{prop}

\begin{proof}
By Proposition \ref{spec_a_smooth},
$\Spec A \backslash 0$ is a smooth scheme, so by definition $\mathscr{Y} = [ (\Spec A \backslash 0) / \mathbb{C}_\lambda^*]$ is a smooth stack, and the second statement is just a restatement of the last proposition. Next, giving a point in $\mathscr{Y}$ over $\mathbb{C}$ is equivalent to giving a map $\mathbb{C}_\lambda^* \to \Spec A\backslash 0$ that is $\mathbb{C}_\lambda^*$-equivariant, which is determined by the image of $1 \in \mathbb{C}_\lambda^*$. If the image is not one of the intersection points, then it is of the form $(u_1, u_2, u_3, y_+, y_-)$ where $y_\pm^2 = f_\pm(u)$. So we have a bijection of $\mathbb{C}$-points of $\mathscr{Y}$ and $Y$ outside of $E_+ \cup E_-$, and a map of smooth schemes bijective on points is an isomorphism, showing that $\mathscr{Y} \to Y$ is a birational map. Finally, $Y$ being the coarse moduli space of $\mathscr{Y}$ follows from the fact that, for a graded $\mathbb{C}$-algebra $R$, the GIT quotient $(\Spec R \backslash 0) // \mathbb{C}^*$ is the coarse moduli space of $[(\Spec R \backslash 0) \,/\, \mathbb{C}^*]$, see for example \cite{gulbrandsen}.
\end{proof}

A common approach to non-commutative algebraic geometry focuses on the abelian categories of coherent sheaves and the corresponding derived categories.

\begin{definition}
\label{def_derived_cat}
 Let $Z$ be an algebraic variety and $\mathcal{B}$ be a locally free sheaf of unital associative algebras on $Z$ of finite rank as $\mathscr{O}_Z$-module. Then the category of coherent sheaves $\Coh-Z, \mathcal{B}$ on the noncommutative algebraic variety $(Z, \mathcal{B})$ is defined to be the category of coherent sheaves of right $\mathcal{B}$-modules on $Z$, and $\mathcal{D}^b(\Coh-Z, \mathcal{B})$ is defined to be the bounded derived category of $\Coh-Z, \mathcal{B}$.
\end{definition}

Of particular interest is the case when $\mathcal{B}$ is a sheaf of Azumaya algebras on $Y$, which encodes a Brauer class. In this case, 
there is another definition of an abelian category based on twisted sheaves, see \cite[Definition 1.2.1]{caldararu2}. Both definitions naturally extend to the case where $Y$ is a Deligne-Mumford stack by pulling back to a scheme cover.

\medskip
We can now state our first main result which is the analog of \eqref{dual-Cliff} for the stack $\mathscr{Y}$ and the K3 surface $X$.
\begin{thm}
\label{main_thm_1}
There exists a sheaf of Azumaya algebras $\mathcal A$ on  $\mathscr{Y}$ so that there is a derived equivalence
\[
\mathcal{D}^b(\Coh-\mathscr{Y},\,\mathcal A) \quad \simeq \quad \mathcal{D}^b(\Coh-X).
\]
\end{thm}

The sheaf  $\mathcal A$ will be constructed in terms of Clifford algebras in Section $3$, and the proof of Theorem \ref{main_thm_1} will be provided there after all relevant definitions are given.

\medskip
Our next goal is to incorporate the Enriques involution on the side of $X$, and its counterpart of the side of $\mathscr{Y}$. We will take inspiration from \cite{borisov}, which suggests the presence of a $(2,2)$-root stack. We can construct another quotient stack with $\mathbb{Z}_2$-stack structure on each of the curves $E_\pm$ except the intersections and $(\mathbb{Z}_2 \times \mathbb{Z}_2)$-stack structure at their intersections. We observe that $-1 \in \mathbb{C}_\lambda^*$ already acts as $(-1, -1)$ on the pair $(y_+,\, y_-)$, so we define the group $G$ as follows:

\begin{definition}
\label{def_root_stack}
Let $G := \mathbb{Z}_2 \times \mathbb{C}_\lambda^*$ and its action on $A$ to be
\[
(\pm 1, \lambda) \cdot (u_0, u_1, u_2, y_+, y_-) = (\lambda^2 u_0, \lambda^2 u_1, \lambda^2 u_2, \lambda^3 y_+, \pm  \lambda^3 y_-).
\]
Let $\widehat{\mathscr Y}$ denote the quotient stack $[\, (\Spec A \backslash 0) \,/\, G\, ]$.
\end{definition}
We observe that the extra $\mathbb{Z}_2$-action together with $-1 \in \mathbb{C}_\lambda^*$ generate a $\mathbb{Z}_2 \times \mathbb{Z}_2$. Note that the selection of the $\mathbb{Z}_2$-action specifies a special role for $y_-$. 

\begin{prop}
\label{stabilizer_root_stack}
The $G$-action on $\Spec A \backslash 0$ has stabilizer group at a point $(u,y_\pm)$:
\[
\begin{cases}
    \text{trivial} \quad &\text{if} \quad y_+,\, y_- \neq 0, \\
    \mathbb{Z}_2 \quad &\text{if} \quad \text{exactly one of } y_+,\, y_- = 0,\\
    \mathbb{Z}_2 \times \mathbb{Z}_2 \quad &\text{if} \quad y_+,\, y_- = 0.
\end{cases}
\]
\end{prop}

\begin{proof}
Assume $(u_0, u_1, u_2, y_+, y_-) \neq 0$ is a fixed point of $(s, \lambda) \in G$. Again this forces $\lambda = \pm 1$. We can then list all the cases:
\[
\begin{cases}
s = 1,\, \lambda = 1 & \text{if } y_+,\, y_- \neq 0, \\
(s,\lambda) = (-1,-1) {\rm ~or~}(1,1)& \text{if } y_+ = 0,\, y_- \neq 0, \\
s = \pm 1,\, \lambda = 1 & \text{if } y_+ \neq 0,\, y_- = 0, \\
s = \pm 1,\, \lambda = \pm 1 & \text{if } y_+,\, y_- = 0.
\end{cases}
\]
\end{proof}

As we have the map $\Spec A \backslash 0 \to W \backslash 0$ and the group actions are compatible, we obtain a map $\widehat{\mathscr Y} \to [ (W \backslash 0) \,/\, \mathbb{C}_t^* ] \cong \mathbb{P}W$ of stacks.

\begin{prop}
\label{prop_root_stack} The stack
$\widehat{\mathscr Y} $ is isomorphic to the $(2,2)$-root stack over $\mathbb{P}W$ branched over $(E_+, E_-)$ in the sense of Cadman \cite{cadman}. In particular, $\widehat{\mathscr Y} \to \mathbb{P}W$ is a birational map.
\end{prop}

\begin{proof}
We will use Cadman's notation in \cite{cadman}, in which the root stack is denoted by $\mathbb{P}W_{(E_+,E_-),\, (2,2)}$. By definition, the root stack $\mathscr{R} := \mathbb{P}W_{(E_+,E_-),\, (2,2)}$ is the fiber product
\[
\xymatrixcolsep{1.5pc}\xymatrixrowsep{2pc}\xymatrix
{
\ & \mathscr{R} \ar[dl] \ar[dr] & \ \\
\mathbb{P}W = [ (W \backslash 0) / \mathbb{C}_t^* ] \ar[dr] & \ & [ \mathbb{C}^2 / (\mathbb{C}_\lambda^*)^2 ] \ar[dl] & \ \\
\ & [ \mathbb{C}^2 / (\mathbb{C}^*)^2 ] & \ & \
}
\]
so it can be identified with $[ \,(W \backslash 0) \times_{\mathbb{C}^2} \mathbb{C}^2\; /\; \mathbb{C}_t^* \times_{(\mathbb{C}^*)^2} (\mathbb{C}_\lambda^*)^2 \, ]$, where the fiber product of the groups is
\[
\mathbb{C}_t^* \times_{(\mathbb{C}^*)^2} (\mathbb{C}_\lambda^*)^2 = \{ (t, \lambda_1, \lambda_2) : t^3 = \lambda_1^2 = \lambda_2^2 \} \subseteq \mathbb{C}_t^* \times \mathbb{C}_\lambda^* \times \mathbb{C}_\lambda^*
\]
via the maps
\[
\xymatrixcolsep{1.5pc}\xymatrixrowsep{2pc}\xymatrix
{
t \in \mathbb{C}_t^* \ar@{|->}[dr] & \ & (\lambda_1, \lambda_2) \in (\mathbb{C}_\lambda^*)^2 \ar@{|->}[dl] \\
\ & (t^3, t^3),\, (\lambda_1^2, \lambda_2^2) \in (\mathbb{C}^*)^2 & \
}
\]
and can be identified with
\[
\left\{ \left( \frac{\lambda_1}{\lambda_2}, \frac{\lambda_1}{t} \right) \right\} = \mathbb{Z}_2 \times \mathbb{C}^* \simeq G.
\]
The maps of the underlying schemes are
\[
\xymatrixcolsep{2pc}\xymatrixrowsep{2pc}\xymatrix
{
 u \in W \backslash 0 \ar@{|->}[dr] & \ & (a_1, a_2) \in \mathbb{C}^2 \ar@{|->}[dl] \\
 \ & (f_+(u), f_-(u)),\, (a_1^2, a_2^2) \in \mathbb{C}^2& \
}
\]
so it is clear that their fiber product is $(W \backslash 0) \times_{\mathbb{C}^2} \mathbb{C}^2 = \Spec A\backslash 0$.
\end{proof}

Our second main result can now be stated.
\begin{thm}
\label{main_thm_2}
There exists a sheaf of Azumaya algebras $\widehat{\mathcal A}$ on  $\widehat{\mathscr Y}$ so that there is a derived equivalence
\[
\mathcal{D}^b(\Coh-\widehat{\mathscr Y},\,\widehat{\mathcal A}) \quad \simeq \quad \mathcal{D}^b(\Coh-X/\sigma).
\]
\end{thm}
The proof of Theorem \ref{main_thm_2} will be provided in Section 3.

\begin{remark}
\label{rmk_diagram_cube}
The relations of the schemes and stacks above (and some additional ones) can be summarized in the diagram below.
\[
\xymatrixcolsep{5pc}\xymatrixrowsep{3pc}\xymatrix
{
\ & \Spec A \backslash 0 \ar[d]_{2:1} \ar[ddr]^{[-/\mathbb{C}_\lambda^*]} & \ \\
\ & [\, (\Spec A \backslash 0) \,/\, \{1\} \times \mathbb{Z}_2\, ] \ar@{.>}[d] \ar[dl]^{2:1} \ar[dr]_{[-/\mathbb{C}_t^*]} & \ \\
[\, (\Spec A \backslash 0) \,/\, \mathbb{Z}_2 \times \mathbb{Z}_2\, ] \ar[d]_{\text{coarse moduli}} \ar[dr]_{[-/\mathbb{C}_t^*]} & \Spec B \backslash 0 \ar@{.>}[dl] \ar@{.>}[dr] & \mathscr{Y} = [\, (\Spec A \backslash 0) \,/\, \{1\} \times \mathbb{C}_\lambda^*\, ] \ar[d]^{\text{coarse moduli}} \ar[dl]^{2:1} \\
W\backslash 0 \ar[dr]_{[-/\mathbb{C}_t^*]} &\widehat{\mathscr Y}= [\, (\Spec A \backslash 0) \,/\, G\, ] \ar[d]_{\text{coarse moduli}} & Y = \Proj B \ar[dl]^{2:1} \\
\ & \mathbb{P}W & \ \\
}
\]
\end{remark}

\section{Clifford Algebras}

In this section, we will define the sheaves of Azumaya algebras $\mathcal A$ and $\widetilde{\mathcal A}$ appearing in Theorem \ref{main_thm_1} and \ref{main_thm_2} in terms of Clifford algebras and provide the proofs to the theorems. 

\medskip
We start by defining several Clifford algebras of interest. In Kuznetsov's paper  \cite{kuznetsov}, the sheaves of Azumaya algebras in question are the even parts of Clifford algebras. In the presence of the involution $\sigma$ in our setup, it is more natural to work with a variant of the full Clifford algebra, rather than its even part. The full Clifford algebra will be denoted by
\[
Cl := Cl(V) = \mathbb{C}[u]\{v_1^+, v_2^+, v_3^+, v_1^-, v_2^-, v_3^-\} \;/\; \langle v_i^\pm v_j^\pm + v_j^\pm v_i^\pm - 2 q_u(v_i^\pm, v_j^\pm),\, v_i^+ v_j^- + v_j^- v_i^+ \rangle.
\]
Its odd and even parts will be denoted by $Cl_{odd}$ and $Cl_{ev}$ respectively, as before. The two smaller Clifford algebras on $V_+$ and $V_-$ over $\mathbb{C}[u]$ will be denoted by
\[
Cl_+ := Cl(V_+) = \mathbb{C}[u]\{v_1^+, v_2^+, v_3^+\} \;/\; \langle v_i^+ v_j^+ + v_j^+ v_i^+ -2 q_u(v_i^+, v_j^+) \rangle,
\]
\[
Cl_- := Cl(V_-) = \mathbb{C}[u]\{v_1^-, v_2^-, v_3^-\} \;/\; \langle v_i^- v_j^- + v_j^- v_i^- - 2 q_u(v_i^-, v_j^-) \rangle.
\]
Note that $Cl$ is the \emph{super tensor product} of $Cl_+$ and $Cl_-$ over $\mathbb{C}[u]$.
For reasons listed in the remark below, we will mainly work with the \emph{ordinary tensor product}
\begin{align*}
\widetilde{Cl} &:= Cl_+ \otimes_{\mathbb{C}[u]} Cl_- = \mathbb{C}[u]\{v_1^+, v_2^+, v_3^+, v_1^-, v_2^-, v_3^-\} \;/\; \langle v_i^\pm v_j^\pm + v_j^\pm v_i^\pm - 2 q_u(v_i^\pm, v_j^\pm),\, v_i^+ v_j^- - v_j^- v_i^+ \rangle.
\end{align*}
In terms of relations, the only difference between $Cl$ and $\widetilde{Cl}$ is the anti-commutativity or commutativity of the variables $v_i^+$ and $v_j^-$, but it is an important difference.

\begin{remark}
\label{rmk_tensor_product}
Ordinary tensor products of Clifford algebras were named quasi Clifford algebras or extended Clifford algebras and studied in \cite{gastineau, marchuk}. There are several reasons why we are working with the ordinary tensor product $\widetilde{Cl}$ and they will be elaborated on in the upcoming propositions.
\begin{enumerate}[(1)]
    \item The ring $A$ is the center of  $\widetilde{Cl}$, while $A$ does not even lie in the center of  the original Clifford algebra $Cl$.    
    \item there is a naturally defined grading on $\widetilde{Cl}$ that is compatible with the $\mathbb{C}_\lambda^*$-grading on $A$, which is essential in the proof of the main theorem,
    \item the ordinary tensor product $\widetilde{Cl}$ is more compatible with $A = A_+ \otimes A_-$ as shown in Proposition \ref{prop_tilde_clifford_center}, and is more readily understood than a Clifford algebra of corank $2$,
    \item consideration of twisted sheaves in Section 4 suggests tensor products of Azumaya algebras.
\end{enumerate}
\end{remark}

On top of the Clifford algebra structure, we will define a new $\mathbb{N}$-grading to handle the grading change between the $\mathbb{C}_\lambda^*$-action and the $\mathbb{C}_t^*$-actions. To handle the case of Enriques surfaces (Theorem \ref{main_thm_2}), we also need to introduce an additional $\mathbb{Z}_2$-grading corresponding to the distinction of $v_i^+$ and $v_i^-$ variables.

\begin{prop}
\label{prop_clifford_grading}
The algebras $Cl$, $Cl_\pm$ and $\widetilde{Cl}$ can be given a $(\mathbb{Z}_2 \times \mathbb{N})$-grading corresponding to the action of $G = \mathbb{Z}_2 \times \mathbb{C}_\lambda^*$ as follows:
\[
\begin{cases}
u_1, u_2, u_3: & \text{degree } (0, 2),\\
v_1^+, v_2^+, v_3^+: & \text{degree } (0, 1), \\
v_1^-, v_2^-, v_3^-: & \text{degree } (1, 1).
\end{cases}
\]
(Here the grading group $\mathbb{Z}_2 = \{0, 1\}$ is in additive notation.) Notice that the $\mathbb{N}$-grading extends the ordinary Clifford grading which counts the parity of the Clifford variables, in the sense that $Cl_{ev}$ consists of elements of even degree in the $\mathbb{N}$-grading, but $\mathbb{C}[u]$ is not contained in the degree $0$ part.
\end{prop}

\begin{proof}
With the relations of the algebras listed above, one can verify that they are homogeneous with respect to the assigned $(\mathbb{Z}_2 \times \mathbb{N})$-degrees, keeping in mind that in the relations $q_u(v_i^\pm, v_j^\pm) = \sum_k u_k q_k(v_i^\pm, v_j^\pm)$ and $q_k(v_i^\pm, v_j^\pm) \in \mathbb{C}$.
\end{proof}

\begin{definition}
\label{def_veronese}
For an $\mathbb{N}$-graded ring $R$, define another $\mathbb{N}$-graded ring $R^{(d)} := \bigoplus_{i \geq 0} R_{d \cdot i}$, and its $i$-th graded piece to be $R_{d \cdot i}$.
\end{definition}

With this notation, the even part $Cl_{ev}$ of the Clifford algebra $Cl$ can be denoted by $Cl^{(2)}$ and so on. We observe  that, as $\mathbb{C}$-vector spaces,
\[
\widetilde{Cl}_{ev} = (Cl_+ \otimes Cl_-)_{ev} = ({Cl_+}_{ev} \otimes {Cl_-}_{ev}) \oplus ({Cl_+}_{odd} \otimes {Cl_-}_{odd}) = Cl_{ev}.
\]
It is not obvious, but crucial for our construction that  they are isomorphic as $\mathbb{C}[u]$-algebras.
\begin{prop}
\label{isomorphism_veronese}
$Cl^{(2)} (= Cl_{ev})$ and $\widetilde{Cl}^{(2)}$ are isomorphic as $\mathbb{C}[u]$-algebras.
\end{prop}

\begin{proof}
We recall that the algebras $Cl^{(2)}$ and $\widetilde{Cl}^{(2)}$ are defined using the $\mathbb{N}$-grading in Proposition \ref{prop_clifford_grading}. We will establish an explicit isomorphism between these two algebras. The generators of $Cl^{(2)}$ over $\mathbb{C}[u]$ are $v_i^+ v_j^+$, $v_i^- v_j^-$ and $v_i^+ v_j^-$. A product of them can be rearranged into the ``standard order'', i.e. so that the $v_i^+$'s are on the left
\[
\prod_{i=1}^m v_i^+ \prod_{j=1}^n v_j^-,
\]
where $m+n$ is even. We construct an isomorphism $\varphi: Cl_{ev} \to (Cl_+ \otimes Cl_-)_{ev}$ by sending the elements
\[
\varphi: \quad \prod_{i=1}^m v_i^+ \prod_{j=1}^n v_j^- \quad \mapsto \quad (-1)^\frac{m(m-1)}{2} \prod_{i=1}^m (v_i^+ \otimes 1) \prod_{j=1}^n (1 \otimes v_j^-).
\]
To verify that it is an algebra homomorphism, we need to show that the image of a product before rearranging
\[
\left( \prod_{i=1}^m v_i^+ \prod_{j=1}^n v_j^- \right) \left( \prod_{i=m+1}^{m+k} v_i^+ \prod_{j=n+1}^{n+\ell} v_j^- \right) \quad \mapsto \quad (-1)^\frac{m(m-1)+k(k-1)}{2} \prod_{i=1}^{m+k} (v_i^+ \otimes 1) \prod_{j=1}^{n+\ell} (1 \otimes v_j^-),
\]
is the same as the image we get if we arrange the product into the standard order beforehand,
\[
(-1)^{nk} \prod_{i=1}^{m+k} v_i^+ \prod_{j=1}^{n+\ell} v_j^- \quad \mapsto \quad (-1)^{nk + \frac{ (m+k)(m+k-1)}{2}} \prod_{i=1}^{m+k} (v_i^+ \otimes 1) \prod_{j=1}^{n+\ell} (1 \otimes v_j^-).
\]
We see that the only difference is the exponent of the scalar $(-1)$, so we check that
\[
\left[nk + \frac{ (m+k)(m+k-1)}{2}\right] - \left[\frac{m(m-1)+k(k-1)}{2}\right] = nk + mk = 0 \hskip -5pt\mod{2},
\]
because $m+n$ is even.
\end{proof}

\begin{remark}
Note that the isomorphism works in the more general context of $\mathbb{Z}/2$-graded algebras. See \cite{deligne2} and \cite[Chapter III, Section 4.6-4.7]{bourbaki} for similar constructions.
\end{remark}

We will now work to prove that $\widetilde{Cl}$ gives a sheaf of Azumaya algebras on $\Spec A\backslash 0$. We first prove that there is an algebra homomorphism $A\to \widetilde{Cl}$
which sends $y_+$ and $y_-$ to the central elements of $Cl_+$ and $Cl_-$ respectively, which can be explicitly computed. The following proposition is key to this claim.
\begin{prop}
\label{prop_clifford_center}
The center of $Cl_+$ is isomorphic to $\mathbb{C}[u][d_+]/\langle d_+^2 - f_+ \rangle$ where
\[
d_+ = v^+_1 v^+_2 v^+_3 - q^+_{32} v^+_1 + q^+_{31} v^+_2 - q^+_{21} v^+_3.
\]
A similar statement holds for $Cl_-$.
\end{prop}

\begin{proof}
We will suppress $()^+$ from our notation for ease of readability. Note that $v_i$ have degree one and $u_i$ have degree $2$ (with $q_{ij}$ being linear forms in $u_i$).
An element in $Cl_+$ of odd degree must be of the form
\[
z = r_0v_1 v_2 v_3 + r_1 v_1 + r_2 v_2 + r_3 v_3
\]
where $r_i \in \mathbb{C}[u]$. It is central if and only if it commutes with $v_1$, $v_2$ and $v_3$.
Then we compute
\begin{align*}
v_1 v_2 v_3 v_1 &= v_1 v_2(2q_{31}-v_1 v_3) = 
2 v_1 v_2 q_{31} - v_1 (2q_{21}-v_1v_2)v_3 = 2v_1 v_2 q_{31} - 2v_1 q_{21} v_3 + q_{11} v_2 v_3,
\\
zv_1 &= r_0v_1 v_2 v_3 v_1 + r_1 q_{11} + r_2 v_2 v_1 + r_3 v_3 v_1
= r_0(2v_1 v_2 q_{31} - 2v_1 q_{21} v_3 + q_{11} v_2 v_3) 
\\&\hskip 30pt + r_1 q_{11} + 2r_2 q_{21} - r_2 v_1 v_2 + 2r_3 q_{31} - r_3 v_1 v_3,
\\
v_1 z &= q_{11}r_0 v_2 v_3 + r_1 q_{11} + r_2 v_1 v_2 + r_3 v_1 v_3,
\\
[z, v_1] &= z v_1 - v_1 z = 2r_0v_1 v_2 q_{31} - 2r_0v_1 q_{21} v_3 + 2r_2 q_{21} - 2r_2 v_1 v_2 + 2r_3 q_{31} - 2r_3 v_1 v_3.
\end{align*}
Setting this to $0$, we get $r_2 = r_0q_{31}$ and $r_3 = -r_0q_{21}$. Consideration of the commutators $[z, v_2]$ and $[z, v_3]$ gives $r_1 = -r_0q_{32}$ and consistent values of $r_1, r_2, r_3$. Thus we see that $z$ is central if and only if 
$$
z\in  (  v_1 v_2 v_3 - q_{32} v_1 + q_{31} v_2 - q_{21} v_3)\C[u].
$$
 Furthermore, it can be verified that for $d= v_1 v_2 v_3 - q_{32} v_1 + q_{31} v_2 - q_{21} v_3$
\begin{align*}
d^2&= v_1 v_2 v_3 v_1 v_2 v_3 - q_{32}v_1 v_2 v_3 v_1 + q_{31} v_1 v_2 v_3 v_2 - q_{21} v_1 v_2 q_{33} \\
&=  - q_{32} q_{11} v_2 v_3 + q_{32}^2 q_{11} - q_{32} q_{31} v_1 v_2 + q_{32} q_{21} v_1 v_3 \\
&\hskip 20pt + q_{31} v_2 v_1 v_2 v_3 - q_{31} q_{32} v_2 v_1 + q_{31}^2 q_{22} - q_{31} q_{21} v_2 v_3 \\
&\hskip 20pt- q_{21} v_3 v_1 v_2 v_3 - q_{21} q_{32} v_3 v_1 - q_{21} q_{31} v_3 v_2 + q_{21}^2 q_{33} \\
&= q_{11} q_{22} q_{33} + q_{12} q_{23} q_{31} + q_{13} q_{21} q_{32} - q_{11} q_{23} q_{32} - q_{12} q_{21} q_{33} - q_{13} q_{22} q_{31} \\
&= \det(q).
\end{align*}
Similarly, for the even part we have 
$$z = r_0+r_1 v_2 v_3 + r_2 v_3 v_1 + r_3 v_1 v_2.$$
Then 
$[v_iv_j,v_k] = 2q_{jk}v_i  -2q_{ik} v_j $ so 
$[z, v_1] = 0$, $[z,v_2]=0$ and $[z,v_3]=0$ give
$$
\left\{
\begin{array}{l}
 r_3 q_{21}-r_2 q_{31}  = 0 \\
r_1 q_{31} - r_3 q_{11} = 0 \\
 r_2 q_{11} -r_1 q_{21} = 0
\end{array}
\right.
~~
\left\{
\begin{array}{l}
 r_3 q_{22}-r_2 q_{32}  = 0 \\
r_1 q_{32} - r_3 q_{12} = 0 \\
 r_2 q_{12} -r_1 q_{22} = 0
\end{array}
\right.
~~
\left\{
\begin{array}{l}
 r_3 q_{23}-r_2 q_{33}  = 0 \\
r_1 q_{33} - r_3 q_{13} = 0 \\
 r_2 q_{13} -r_1 q_{23} = 0
\end{array}
\right.
$$
respectively. We can work in the field of fractions $\C(u_1,u_2,u_3)$ and observe that the above equations imply that $(r_1,r_2,r_3)$ is proportional to 
$$
(q_{1i},q_{1i},q_{1i}) 
$$
for all $i$, with coefficient in $\C(u_1,u_2,u_3)$. However, the columns of the matrix $(q_{ij})$ are not proportional to each other, because its determinant is nonzero.
Thus, $(r_1,r_2,r_3)$ must be $(0,0,0)$, and we see that the even part of the center of $Cl_+$ is $\C[u]$.
\end{proof}

\begin{prop}
\label{prop_tilde_clifford_center}
The center of $\widetilde{Cl}$ is naturally isomorphic to $A$, with $d_\pm$ identified with $y_\pm$.
\end{prop}

\begin{proof}
This simply follows from Proposition  \ref{prop_clifford_center} and the fact that both $Cl_+$, $Cl_-$ are their centers are free $\C[u]$-modules. We consider the  bases of $\Cl_\pm$ made of non-repeating monomials in $v_i^\pm$, which we denote by $v_I^+$ and $v_I^-$ where $I,J$ are subsets of the index set $\{1,2,3\}$. Then for any central element of the tensor product 
$$
z=\sum_{I,J\subseteq \{1,2,3\}}  g_{I,J}(u)
v_I^+ \otimes v_J^-
$$
the vanishing of the commutators with $Cl_+\otimes 1$ implies that for each $J$ the elements $\sum_{I} g_{I,J} v_I^+$ are in $\C[u] + d^+\C[u]$. We can then rewrite
$z$ as
$$
z = 1\otimes \sum_{J \subseteq \{1,2,3\}}  g_{\emptyset,J}(u)  V_J^- + d_+\otimes  \sum_{J} h_J(u)  V_J^-.
$$
When we take a commutator with elements of $1\otimes Cl_-$, and use linear independence of $1$ and $d_+$, we see that both sums in the above equations must be central in $Cl_-$.
\end{proof}

\begin{remark}
\label{rmk_clifford_embedding}
We cannot embed $A$ in the full Clifford algebra $Cl$ because $y_+$ and $y_-$ commute in $A$, while $d_+$ and $d_-$ do not commute in $Cl$. This is why we must use the ordinary tensor product $\widetilde{Cl}$ of $Cl_+$ and $Cl_-$.
\end{remark}

\begin{prop}
\label{prop_azumaya_over_a}
The sheaf that corresponds to  $\widetilde{Cl}$  is a sheaf of  Azumaya algebras over $\Spec A\backslash 0$.
\end{prop}

\begin{proof}
First, we look at the diagram of algebra inclusions, in which each small diamond is a tensor product:
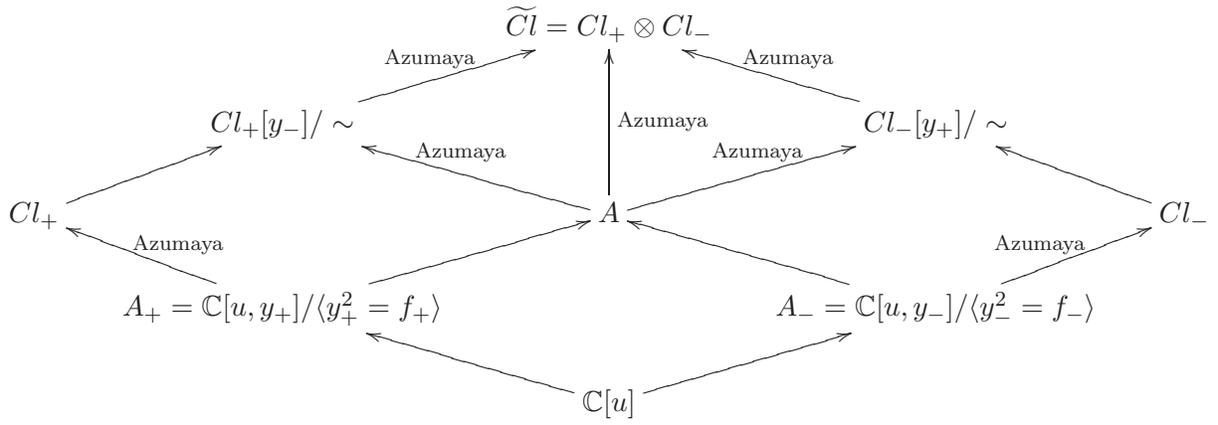
\[
\xymatrixcolsep{1.5pc}\xymatrixrowsep{1.5pc}\xymatrix
{
\ & \ & \widetilde{Cl} = Cl_+ \otimes Cl_- & \ & \ \\
\ & Cl_+[y_-]/\sim \ar[ur]^{\text{Azumaya}} & \ & Cl_-[y_+]/\sim \ar[ul]_{\text{Azumaya}} & \ \\
Cl_+ \ar[ur] & \ & A \ar[uu]_{\text{Azumaya}} \ar[ul]_{\text{Azumaya}} \ar[ur]^{\text{Azumaya}} & \ & Cl_- \ar[ul] \\
\ & A_+ = \mathbb{C}[u, y_+]/\langle y_+^2 - f_+\rangle \ar[ul]_{\text{Azumaya}} \ar[ur] & \ & A_- = \mathbb{C}[u, y_-]/\langle y_-^2 - f_-\rangle \ar[ul] \ar[ur]^{\text{Azumaya}} & \ \\
\ & \ & \mathbb{C}[u] \ar[ul] \ar[ur] & \ & \ \\
}
\]
We will prove that the maps labelled by ``Azumaya" give sheaves Azumaya algebras away from the origin $u=0$ in $\C^3$. We use the fact that:
\begin{enumerate}[(i)]
    \item the pullback of an Azumaya algebra is Azumaya, and
    \item the tensor product of two Azumaya algebras is also Azumaya.
\end{enumerate}
So in order to show that $A \to \widetilde{Cl}$ is Azumaya, we just need to show it for the maps $A_+ \to Cl_+$ and $A_- \to Cl_-$ at the lower left and right. It will be proved in a similar fashion as proposition 3.13 in Kuznetsov's paper \cite{kuznetsov}.\\
\\
As being an Azumaya algebra is a local property, it suffices to check that the fiber of $Cl_+$ at each point of $\Spec A_+ \backslash 0$ is a matrix algebra. On the algebra level, $y_+$ is sent to $d_+$. For a point $(u, y_+) \in \Spec A_+ \backslash 0$ such that $y_+ \neq 0$, taking the fiber of $Cl_+$ means fixing $u$ and $y_+$. For fixed $u$, the quadratic form $q^+_u$ is of full rank and it is well-known that fiber $Cl(q^+_u)$ is a product of two rank $2$ matrix algebras, which comes from its two irreducible Clifford modules classified by the action of $d_+$, so fixing $y_+ = d_+$ means choosing one of the two components. For points of the form $(u, 0)$, the quadratic form $q^+_u$ is of corank $1$ and the maximal ideal at $(u,0)$ is generated by $y_+$, so the fiber of $Cl_+$ is $Cl(q^+_u)/\langle d_+\rangle$ which is isomorphic to a matrix algebra of rank $2$.
\end{proof}

Now we have all the tools to carry out the main step of the proofs of Theorems \ref{main_thm_1} and \ref{main_thm_2}.
Let $\mathcal A$ be the sheaf of algebras associated to $\widetilde{Cl}$ on $\mathscr{Y}$ with the $\mathbb{N}$-grading corresponding to the $\mathbb{C}_\lambda^*$-action. The idea is to connect the sheaf $\mathcal A$ to the one constructed in \cite{kuznetsov} by utilizing a Veronese-type theorem from \cite{verevkin}.

\medskip
Let $R$ be a $\mathbb{N}$-graded left Noetherian ring which is generated by $R_1$ as a $R_0$-algebra. We denote by $\Cohproj(R)$ the localization of
the category of finitely generated graded right modules over $R$ by the subcategory modules which are finitely generated over $R_0$. 
\begin{lem}\cite{verevkin}
\label{lem_veronese}
Then the categories $\Cohproj(R^{(d)})$ are equivalent for any natural number $d$. 
\end{lem}

An immediate consequence of Lemma \ref{lem_veronese} is that it connects our setting to Kuznetsov's main theorem, which is included here for readers' benefit.

\begin{prop}
\label{main_thm_kuznetsov}\cite{kuznetsov} Let $V = \mathbb{C}^6$ with coordinates $x_1, \ldots, x_6$ and $W \subseteq S^2V^\vee$ with coordinates $u_1, u_2, u_3$ a dimension $3$ subspace of rank $\geq 4$ (except the origin) homogeneous quadratic forms on $V$ such that the intersection $X$ of quadrics parametrized by $W$ is a complete intersection. Let $q \in \mathbb{C}[u_i, x_j]$ be the quadratic form corresponding to the choice of $W$. Let $Y$ be the double cover of $\mathbb{P}W$ ramified over the sextic curve of degenerate quadrics, and $Cl(V)$ be the sheaf of Clifford algebras on $Y$ associated to $q$. Then we have derived equivalences:
\[
\mathcal{D}^b(\Coh-X) \quad \simeq \quad \mathcal{D}^b(\Coh-\mathbb{P}W,\, Cl(V)_{ev}) \quad \simeq \quad \mathcal{D}^b(\Coh-Y,\, Cl(V)_{ev}).
\]
Here $Cl(V)_{ev}$ is equipped with the $\mathbb{C}_t^*$-grading and $Cl(V)$ with the $\mathbb{C}_\lambda^*$-grading.
\end{prop}

\emph{Proof of Theorem \ref{main_thm_1}.}
Lemma \ref{lem_veronese} and Proposition \ref{isomorphism_veronese} tells us that the categories
$$
D^b( \Coh-\mathbb{P}W,\, Cl_{ev})=\Cohproj (Cl_{ev})\simeq \Cohproj (\widetilde{Cl}_{ev})$$ 
and
$$
\quad \Cohproj (\widetilde{Cl}) = \quad D^b(\Coh-\mathscr{Y},\, \widetilde{Cl})
$$
are equivalent, which translates to the equivalence of $\Cohproj (\widetilde{Cl}_{ev})$ and $\Cohproj( \widetilde{Cl})$. Here we used 
This allows us to apply Proposition
\ref{main_thm_kuznetsov} to conclude that
\[
D^b(\Coh-X) \quad \simeq \quad D^b( \Coh-\mathbb{P}W,\, Cl_{ev}) \quad \simeq \quad D^b(\Coh-\mathscr{Y},\, \widetilde{Cl}).
\]
\hfill$\Box$

Let $\widehat{\mathcal A}$ be the sheaf of algebras associated to $\widetilde{Cl}$ on $\widehat{\mathscr Y}$ with the $(\mathbb{Z}_2 \times \mathbb{N})$-grading corresponding to the $G$-action.

\medskip
\emph{Proof of Theorem \ref{main_thm_2}.}
Let $\widehat{\mathcal A}$ be the sheaf of algebras associated to $\widetilde{Cl}$ on $\widehat{\mathscr Y}$ with the $(\mathbb{Z}_2 \times \mathbb{N})$-grading corresponding to the $G$-action. Considering the induced $\mathbb{N}$-grading on $\widetilde{Cl}$, we can again apply Lemma \ref{lem_veronese} to conclude that the equivalence of categories $\Cohproj (\widetilde{Cl}_{ev})$ and $\Cohproj (\widetilde{Cl})$. We can then use \cite[Theorem 6.2]{borisov}  to get
\[
D^b(\Coh-(X/\sigma)) \simeq D^b( \Coh-[\mathbb{P}W / \mathbb{Z}_2],\, Cl_{ev}) \simeq D^b( \Coh-[\mathbb{P}W / \mathbb{Z}_2],\, \widetilde{Cl}_{ev}) \simeq D^b(\Coh-(\mathscr{Y}/\mathbb{Z}_2),\, \widetilde{Cl}).
\]
\hfill$\Box$

\begin{remark}
\label{rmk_element_h}
In \cite[Section 9.2]{borisov}, a semi-direct product of a Clifford algebra and $\mathbb{C}[h]/\langle h^2-1 \rangle$ is considered. This corresponds to the extra $\mathbb{Z}_2$-grading on the sheaf of algebras.
\end{remark}

\section{Severi-Brauer Varieties}

There are several formulations of the Brauer group and its elements Brauer classes on a scheme. One way is to define them as Azumaya algebras under Morita equivalence, another way is to represent them by twisted sheaves \cite{caldararu2}, whose projectivizations are projective bundles that are locally trivial in the \'{e}tale topology, namely the Severi-Brauer varieties. We will now construct projective bundles on $\mathscr{Y}$ and $\widehat{\mathscr Y}$ and show that they represents the Azumaya algebra $\widetilde{Cl}$. This gives us the connection between the Clifford algebra and the base space. We will work on $\Spec A \backslash 0$ and our constructions will be equivariant with respect to the action of the group $G$.

\medskip
To construct the Brauer class on $\Spec A \backslash 0$  as a projective bundle and to connect it to the Clifford algebra involves several steps:
\begin{enumerate}[(1)]
    \item Construct two $\mathbb{P}^1$-bundles on $\Spec A_\pm \backslash 0$ from the data of $q^\pm$.
    \item Construct a $\mathbb{P}^3$-bundle on $\Spec A \backslash 0$ from the two $\mathbb{P}^1$-bundles.
    \item Exhibit an isomorphism from the projectivization of the Clifford module to the $\mathbb{P}^1$-bundle.
\end{enumerate}

In the second half of this section, we will also relate the $\mathbb{P}^3$-bundle in our case to the $\mathbb{P}^3$-bundle that comes from Kuznetsov's construction which considers the quadratic form $q$ as a whole.

\medskip
Recall from Section 2 that $W$ parametrizes quadratic forms $q$ and each $q$ can be decomposed into $q = q^+ + q^-$.  Consider the conic fibration $C$ over $W \backslash 0$ whose fiber over a point $u \in W \backslash 0$ is the conic $\{q_u^+ = 0\} \subseteq \mathbb{P}V_+$. Because we restrict to a general family of such quadratic forms $q$, we can assume $C$ is smooth by Bertini's Theorem, because $C$ is a hyperplane of degree $(1,2)$ in $(W \backslash 0) \times \mathbb{P}V_+$. These are nonsingular conics for $u \in W \backslash \cone (E_+)$ so we have a generic $\mathbb{P}^1$-fibration, while the fibers over $u \in \cone (E_+)$ are pairs of lines. It is not possible to somehow replace the pairs of lines by a smooth $\P^1$ over $W\backslash 0$, but it becomes possible to do so on the double cover
$\Spec A_+\backslash 0$.
We will present 2 ways to construct such $\mathbb{P}^1$-bundles.

\medskip
First method is the following: take the conic fibration, denoted $\adj(C)$, over $W \backslash 0$ whose fiber over a point $u \in W \backslash 0$ is the conic $\{\adj(q_u^+) = 0\} \subseteq \mathbb{P}V_+^\vee$ in the dual projective space, where $\adj(q_u^+)$ denotes the quadratic form whose associated matrix is the adjugate matrix of $q_u^+$. The general fibers are then dual conics and the fibers over $u \in E_+$ become double lines. Over $W \backslash \cone (E_+)$, we have a birational map $C \to\adj(C)$ arising from the dual curve map. For convenience, denote the pullback of $\adj(C)$ to the ramified double cover $\Spec A_+ \backslash 0$ by
\[
\adj(C_+) := \adj(C) \times_{(W\backslash 0)} (\Spec A_+ \backslash 0).
\]
Take the normalization $P_+$ of $\adj(C_+)$, then the composition $P_+ \to \Spec A_+ \backslash 0$ will be a $\mathbb{P}^1$-bundle, as the next proposition shows.

\begin{prop}
\label{prop_p1_bundle_normalization}
The variety $P_+$ is a $\mathbb{P}^1$-bundle over $\Spec A_+ \backslash 0$.
\end{prop}

\begin{proof}
To see this, restrict to a Zariski neighborhood $U$ of a point $u_0$ lying over $\cone (E_+)$ in $\Spec A_+ \backslash 0$. The quadratic forms can be simultaneously diagonalized to have the form (abusing notations):
\[
q_u^+ = \begin{bmatrix}
a(u) & 0 & 0 \\
0 & b(u) & 0 \\
0 & 0 & \epsilon(u)
\end{bmatrix}
\]
where $\epsilon(u_0) = 0$ and $a(u)$, $b(u) \neq 0$ for $u \in U$, i.e. $a$ and $b$ are units in $\mathbb{C}[U]$. Its adjugate matrix is
\[
\adj(q_u^+) = \begin{bmatrix}
b(u)\epsilon(u) & 0 & 0 \\
0 & a(u)\epsilon(u) & 0 \\
0 & 0 & a(u)b(u)
\end{bmatrix}
\]
so at $u_0$ it is
\[
\adj(q_{u_0}^+) =
\begin{bmatrix}
0 & 0 & 0 \\
0 & 0 & 0 \\
0 & 0 & a(u_0)b(u_0)
\end{bmatrix}
\]
and in the coordinates $x_1^+, x_2^+, x_3^+$ of $V_+$ the fiber is the double line $(x_3^+)^2 = 0$.

\medskip
Now we verify that the normalization is smooth. By picking complementary coordinates $u_1$, $u_2$ to the local coordinate $\epsilon$, we identify the coordinate ring $\mathbb{C}[U]$ with $\mathbb{C}[u_1, u_2, \epsilon]$. The coordinate ring of $\adj(C)$ over $U$ is thus
\[
\mathbb{C}[U][y_+, x_1^+, x_2^+, x_3^+]/ \langle\, y_+^2 - ab\epsilon,\, b\epsilon(x_1^+)^2 + a\epsilon(x_2^+)^2 + ab (x_3^+)^2 \,\rangle.
\]
We observe that
\[
\left( \frac{x_3^+}{y_+}  \right)^2 = \frac{(x_3^+)^2}{(ab\epsilon)^2} = -\frac{(x_1^+)^2}{a^2b} -\frac{(x_2^+)^2}{ab^2}
\]
belongs to the coordinate ring but $x_3^+/y_+$ does not. We adjoin $\widetilde{x_3}^+ = x_3^+/y_+$ and show that the resulting ring
\[
\mathbb{C}[U][y_+, x_1^+, x_2^+, \widetilde{x_3}^+] \bigg/ \left\langle\, y_+^2 - ab\epsilon,\, \frac{(x_1^+)^2}{a^2b} + \frac{(x_2^+)^2}{ab^2} + (\widetilde{x_3}^+)^2 \,\right\rangle
\]
is regular. The Jacobian matrix of partial derivatives of the relations with respect to $u_1$, $u_2$, $\epsilon$, $y_+$, $x_1^+$, $x_2^+$, $\widetilde{x_3}^+$ is the following.
\[
\begin{bmatrix}
-\epsilon\frac{\partial (ab)}{\partial u_1} & -\epsilon\frac{\partial (ab)}{\partial u_2} & -ab-\epsilon\frac{\partial (ab)}{\partial \epsilon} & 2y_+ & 0 & 0 & 0 \\
* & * & * & 0 & \frac{2x_1^+}{a^2b} & \frac{2x_2^+}{ab^2} & 2\widetilde{x_3}^+
\end{bmatrix}
\]
Since $x_1^+, x_2^+, \widetilde{x_3}^+$ are coordinates on $\mathbb{P}V_+$ and cannot be all zero, the matrix has rank $1$ precisely when the first row is a multiple of the second row. This forces it to be the zero row, so $y_+ = 0$ and thus $\epsilon = 0$ by the relation, but then $ab \neq 0$ in the third entry. So there is no point at which the matrix has rank $1$.

\medskip
Since all the fibers of $P_+$ are smooth conics given by $\frac{(x_1^+)^2}{a^2b} + \frac{(x_2^+)^2}{ab^2} + (\widetilde{x_3}^+)^2 = 0$, we obtain a $\mathbb{P}^1$-bundle.
\end{proof}

In the second method, we start from the pullback
\[
C_+ := C \times_{(W\backslash 0)} (\Spec A_+ \backslash 0)
\]
of $C$ to $\Spec A_+ \backslash 0$, and consider the blowup $P_{1,+}$ of C$_+$ along its singular locus.

\begin{prop}
\label{prop_p1_bundle_blowup}
The singular locus of $C_+$ consists of the singular points in the singular fibers over $\cone(E_+)$, and the blowup $P_{1,+}$ of $C_+$ along the singular locus is smooth.
\end{prop}

\begin{proof}
Observe that $C_+$ is cut out by $y_+^2 - f_+$ and $q^+$ inside $(\Spec A_+ \backslash 0) \times \mathbb{P}V_+$. Taking partial derivatives with respect  to the variables $u_1$, $u_2$, $u_3$, $y_+$, $x_1^+$, $x_2^+$, $x_3^+$, we get the matrix
\[
\begin{bmatrix}
-\frac{\partial f_+}{\partial u_1} & -\frac{\partial f_+}{\partial u_2} & -\frac{\partial f_+}{\partial u_3} & 2y_+ & 0 & 0 & 0 \\
\frac{\partial q^+}{\partial u_1} & \frac{\partial q^+}{\partial u_2} & \frac{\partial q^+}{\partial u_3} & 0 & \frac{\partial q^+}{\partial x_1^+} & \frac{\partial q^+}{\partial x_2^+} & \frac{\partial q^+}{\partial x_3^+}
\end{bmatrix}.
\]
Note that one of $\frac{\partial f_+}{\partial u_i}$ must be nonzero because $E_+ = \{f_+ = 0\}$ is smooth, and also it is assumed that $C_+ = \{q^+ = 0\}$ is smooth. If $y_+ \neq 0$, then the Jacobian must have rank $2$. If $y_+ = 0$, then setting rank $\leq 1$ implies that second row is a multiple of the first row, so $\frac{\partial q^+}{\partial x_1^+} = \frac{\partial q^+}{\partial x_2^+} = \frac{\partial q^+}{\partial x_3^+} = 0$. So the singular locus of $C_+$ is contained in the singular loci of the singular fibers of $C_+$.

\medskip
In the other direction, we likewise locally diagonalize $q^+$ as $q^+ = r_1(u) (x_1^+)^2 + r_2(u) (x_2^+)^2 + r_3(u) (x_3^+)^2$ where $r_1$, $r_2$ are invertible and $r_3$ is a local coordinate. Pulling back to $\Spec A_+ \backslash 0$, we have $r_3 = y_+^2 / r_1 r_2$, so we see that $C_+$ is singular when $y_+ = x_1^+ = x_2^+ = 0$. (Let $u_1$, $u_2$ be complementary coordinates to $r_3$, and note that $q^+$ lies in the square of the maximal ideal.) Hence we have the reverse inclusion.

\medskip
Since $(x_1^+, x_2^+, x_3^+) \neq (0,0,0)$, we see that $x_3^+ \neq 0$ at the singular points. The form of $q^+$ shows that we have an $A_1$-singularity along the singular locus. Therefore the blowup $P_{1,+}$ of the singular locus is resolution of $C_+$.
\end{proof}

Note that for smooth conics in $\mathbb{P}V_+$ over $W\backslash \cone(E_+)$, there is a natural map of dual conics mapping a point to the tangent line at that point and it gives a rational map $C_+ \to \adj(C_+)$, which fiber-wise it maps $(x_1^+:x_2^+:x_3^+)$  to $\left( \frac{\partial q^+}{\partial x_1^+} : \frac{\partial q^+}{\partial x_2^+} : \frac{\partial q^+}{\partial x_3^+} \right)$. We claim that it extends to a morphism $P_{1,+} \to C_+$, hence factors through the normalization $P_{1,+} \to P_+ \to \adj(C_+)$. Moreover, let $H_0 \subseteq C_+$ denote the pullback of $\cone(E_+) \subseteq \Spec A_+ \backslash 0$, which is fiber-wise the singular quadric $\{q^+ = 0\}$, and $H \subseteq P_{1,+}$ be the strict transform of $H_0$ in $P_{1,+}$, then we have the following.

\begin{prop}\label{above}
\label{prop_p1_bundle_morphism}
There exists a morphism $P_{1,+} \to \adj(C_+)$, and fiber-wise it contracts the two lines in $H \subseteq P_{1,+}$ into two points.
\end{prop}

\begin{remark}
\label{rmk_bundle_diagram}
Before proving Proposition \ref{above}, we summarize the two constructions  in the following diagram.
\[
\xymatrixcolsep{3pc}\xymatrixrowsep{2pc}\xymatrix
{
\ & P_{1,+} \ar[dr]_{\text{blow up}} \ar[ddl]_{\text{blow up}} & \ \\
\ & \ & P_+ \ar[d]_{\text{normalization}} \\
C_+ \ar@{-->}[rr]_{\text{birational}} \ar[dr] \ar[ddd] & \ & \adj(C_+) \ar[dl] \ar[ddd] \\
\ & \Spec A_+ \backslash 0 \ar[d] & \ \\
\ & W\backslash 0 & \ \\
C \ar@{-->}[rr]_{\text{birational}} \ar[ur] & \ & \adj(C) \ar[ul]
}
\]
\end{remark}

\begin{proof}
Restricting to a Zariski neighborhood $U$ of a point $u_0$ lying over $\cone(E_+)$ in $\Spec A_+ \backslash 0$, the quadratic forms can be simultaneously diagonalized to be $q^+ = r_1(u)(x_1^+)^2 + r_2(u)(x_2^+)^2 + r_3(u) (x_3^+)^2$
where $r_1$, $r_2$ are invertible and $r_3$ is a local coordinate. In this case the equations of the singular locus in $C_+$ reduce to
\[
y_+ = 0, \quad x_1^+ = 0, \quad x_2^+ = 0.
\]
Since $(x_1^+, x_2^+, x_3^+) \neq (0,0,0)$, we see that $x_3^+ \neq 0$ at $u_0$ and we can assume it is invertible in $U$. The blowup is locally $\Proj (\mathscr{O}_{C_+} + \mathcal{I} + \mathcal{I}^2 + \dots)$ where $\mathcal{I} = \langle y_+, x_1^+, x_2^+ \rangle$ is the ideal sheaf. Let $\widetilde{y_+}$, $\widetilde{x_1}^+$, $\widetilde{x_2}^+$ denote the corresponding sections in the degree $1$ piece of the Rees algebra $\mathscr{O}_{C_+} + \mathcal{I} + \mathcal{I}^2 + \dots$, then the exceptional divisor has equation 
\[
(\widetilde{x_1}^+)^2 + (\widetilde{x_2}^+)^2 + (x_3^+)^2 (\widetilde{y_+})^2 = 0 \quad (\text{see proof of Proposition } \ref{prop_p1_bundle_blowup}.).
\]
Then we see that the map $\left( \frac{\partial q^+}{\partial x_1^+} : \frac{\partial q^+}{\partial x_2^+} : \frac{\partial q^+}{\partial x_3^+} \right)$ extends to the formula
\[
\left( 2\widetilde{x_1}^+ : 2\widetilde{x_2}^+ : 2x_3^+ y_+ \widetilde{y_+} \right)
\]
to the exceptional divisor in the new coordinates. which is well-defined because the first two components cannot be simultaneously $0$: if $\widetilde{x_1}^+ = \widetilde{x_2}^+ = 0$ then $\widetilde{y_+} = 0$ from the equation of the exceptional divisor, but $(\widetilde{y_+}, \widetilde{x_1}^+, \widetilde{x_2}^+) \neq (0,0,0)$.

\medskip
Next, we see that over a point $u \in \cone(E_+)$ the quadric $q^+$ degenerates to $2$ lines passing through the origin, so the points on $q^+$ has two fixed ratios $\widetilde{x_1}^+ : \widetilde{x_2}^+$ for a given $u$. The map sends these points to $( 2\widetilde{x_1}^+ : 2\widetilde{x_2}^+ : 0 )$, i.e. fiber-wise $H$ is contracted to $2$ points.
\end{proof}

We have described the $\mathbb{P}^1$-bundle $P_+$ associated to $q^+$ and $E_+$ on $\Spec A_+ \backslash 0$. Likewise there is a $\mathbb{P}^1$-bundle $P_-$ associated to $q^-$ and $E_-$ on $\Spec A_- \backslash 0$. We can now construct a $\mathbb{P}^3$-bundle on $\Spec A \backslash 0$ from $P_+$ and $P_-$. We denote the pullbacks of $P_\pm$ from $\Spec A_\pm \backslash 0$ to $\Spec A \backslash 0$ by $P_\pm^*$.

\begin{prop}
\label{prop_p3_bundle}
There exists a Severi-Brauer variety $P$ over $\Spec A \backslash 0$ such that $P_+^*$ and $P_-^*$ are fiber-wise embedded in $P$ via the Segre embedding $\mathbb{P}^1 \times \mathbb{P}^1 \xhookrightarrow{} \mathbb{P}^3$, and $P$ corresponds to the Brauer class represented by $\widetilde{Cl}$.
\end{prop}

\begin{proof}
We will exhibit two ways to construct the $\mathbb{P}^3$-bundle. First we can take the twisted sheaves $\mathscr{P}_\pm$ associated to $P_\pm$ \cite{caldararu2} (unique up to twisting of a line bundle). We can then pull back $\mathscr{P}_\pm$ from $\Spec A_\pm \backslash 0$ to $\Spec A \backslash 0$ to get $\mathscr{P}_+^*$, and their tensor product $\mathscr{P}_+^* \otimes \mathscr{P}_-^*$ is then a twisted sheaf on $\Spec A \backslash 0$ whose fiber is $\mathbb{C}^4$, hence its projectivization $P'$ becomes a $\mathbb{P}^3$-bundle. By \cite{caldararu2, kollar}, $P'$ corresponds to the Brauer class represented by $\widetilde{Cl} = Cl_+^* \otimes Cl_-^*$. From this description it is not obvious that we get an algebraic variety because we went through the \'{e}tale or analytic topology having twisted sheaves involved, but we observe that fiber-wise $P_+$ and $P_-$ are embedded in $P'$ via the Segre embedding $\mathbb{P}^1 \times \mathbb{P}^1 \xhookrightarrow{} \mathbb{P}^3$, which allows us to show that $P'$ is algebraic.

\medskip
To do that, we make use of the fact that both $P_\pm \to \Spec A_\pm \backslash 0$ are projective by construction, so are their pullbacks $P_\pm^*$ to $\Spec A \backslash 0$, and their fiber product $P_+^* \times P_-^*$ there. By definition we can embed it in a trivial $\mathbb{P}^N$-bundle over $\Spec A \backslash 0$. Now to every $(1,1)$-divisor in the fiber we can associate its image in $\mathbb{P}^N$, which is a closed subscheme in $\mathbb{P}^N$. This assignment gives a map from $P_+^* \times P_-^*$ to the Hilbert scheme $\Hilb (\mathbb{P}^N)$. Since there is a $\mathbb{P}^3$-family of $(1,1)$-divisors for each fiber, the closure of the image of $P_+^* \times P_-^*$ in $\Hilb (\mathbb{P}^N)$ will form an algebraic $\mathbb{P}^3$-bundle $P''$. We see that $P_+^* \times P_-^* \to P''$ is fiber-wise a Segre embedding so it must agree with $P'$, as both constructions are canonical. Hence we obtain the desired algebraic $\mathbb{P}^3$-bundle $P$, which is a Severi-Brauer variety.
\end{proof}

Next we will exhibit an explicit relation of the Brauer class of $\widetilde{Cl}$ and the $\mathbb{P}^3$-bundle $P$ when formulated via the quadric fibrations. We state here \cite[Theorem 1.3.5]{caldararu2} that serves as an equivalent definition of an Azumaya algebra:

\begin{thm}
\label{thm_azumaya_def}
Let $\mathscr{A}$ be an Azumaya algebra over $X$, and let $\alpha \in \Br'(X)$ be the element that $\mathscr{A}$ represents. Then there exists a locally free $\alpha$-twisted sheaf $\mathscr{E}$ of finite rank (not necessarily unique) such that $\mathscr{A}$ is isomorphic to the sheaf of endomorphism algebra of $\mathscr{E}$. Conversely, for any $\alpha \in \Br'(X)$ such that there exists a locally free $\alpha$-twisted sheaf of finite rank, the sheaf of endomorphism algebra of $\mathscr{E}$ is an Azumaya algebra whose class in $\Br'(X)$ is $\alpha$.
\end{thm}

So it suffices to show the following.

\begin{prop}
\label{prop_p1_bundle_clifford_module}
The bundle
$P_+$ is Zariski locally isomorphic to the projectivization of an irreducible Clifford module on $\Spec A_+ \backslash \cone(E_+)$. Similar statement holds for $P_-$.
\end{prop}

\begin{proof}
Let $M$ be an irreducible $Cl_+$-module of rank $2$, when $Cl_+$ is considered a sheaf of algebras on $\Spec A_+ \backslash \cone(E_+)$. We want to get an isomorphism $\mathbb{P}M \to P_+$. Fix $m \neq 0$ of $M$ and consider $\mathbb{C}m \in \mathbb{P}M$. The map $V_+ \to M$ sending $v \mapsto v \cdot m$ has a kernel by counting dimension. The annihilator of $m$ consists of isotropic vectors because $0 = vvm = q^+(v,v)m$ and $m \neq 0$. So for $q^+$ of full rank, the annihilator must be $1$-dimensional because $\{q^+=0\}$ doesn't contain any line in $\mathbb{P}V_+$. This gives a natural map $\mathbb{P}M \to \{q^+=0\} \subseteq \mathbb{P}V_+$ sending $\mathbb{C}m \mapsto \Ann(m)$. For the inverse, send $\mathbb{C}v \in \{q^+ = 0\} \subseteq \mathbb{P}V_+$ to
\[
\ker(v) = \{m \in M : v \cdot m = 0\}.
\]
We need to make sure $\ker(v)$ is always $1$-dimensional to get an element in $\mathbb{P}M$. Since $q^+$ is assumed of full rank, $q^+(v,v) \neq 0$. If $\ker(v) = M$, take $v' \neq v$ in $\{q^+ = 0\} \subseteq \mathbb{P}V_+$ and $m' \in M$ linearly independent with $m$ such that $v'm' = 0$, then $\Ann(m')$ contains both $v$ and $v'$, a contradiction to $\dim \Ann(m') = 1$. Hence we have a map $\{q^+ = 0\} \to \mathbb{P}M$ sending $\mathbb{C}v \mapsto \ker(v)$. It is then evident that the two maps are inverses to each others.
\end{proof}

Over the points $u \in W \backslash \cone(E_+ \cup E_-)$, we are in a similar situation as in Kuznetsov's case for which the Clifford algebras in the fiber are of full rank, so we can make comparison of the $\mathbb{P}^3$-bundles. First we give a description to the $\mathbb{P}^3$-bundle in Kuznetsov's case: when $q$ is a quadratic form in $\mathbb{P}V $ of full rank, in the quadric $\{q = 0\} \subseteq (W\backslash 0) \times \mathbb{P}V$ there are two families of rulings, parametrized by $\mathbb{P}^3$, of isotropic subspaces isomorphic to $\mathbb{P}^2$'s. These two $\mathbb{P}^3$ families naturally form two $\mathbb{P}^3$-bundles $P'$ and $P''$ over $W \backslash \cone(E_+ \cup E_-)$. Note that $P'$ is isomorphic to $P''$.

\begin{prop}
\label{prop_kuznetsov_clifford_module}
$P'$ and $P''$ are isomorphic to the projectivizations of the two irreducible modules of $Cl_{ev}$ over $W \backslash \cone(E_+ \cup E_-)$.
\end{prop}

\begin{proof}
When restricted to $u \in W\backslash \cone(E_+ \cup E_-)$, the rank of the quadratic form is $6$, the Clifford algebra $Cl_{ev}$ has two irreducible modules $M_0$ and $M_1$ of rank $4$, and $M = M_0 \oplus M_1$ is the irreducible module for $Cl$. The embedding $V \subseteq Cl$ gives rise to two maps $V \otimes M_0 \to M_1$ and $V \otimes M_1 \to M_0$ via the Clifford action of $Cl$ on $M$.

\medskip
Fix $m \neq 0$ of $M_0$ and consider $\mathbb{C}m \in \mathbb{P}M_0$. The map $V \to M_1$ sending $v \mapsto v \cdot m$ has a matrix of rank $3$ and hence a kernel $\Ann(m)$ of dimension $3$. To see this, we may identify the quadratic form with a natural pairing on $U \oplus U^\vee$ where $U = \mathbb{C}^3$, then the Clifford algebra and modules are built from wedge products and contractions \cite{deligne}. Fix bases $\{v_1, v_2, v_3\}$ for $U$, $\{v_4, v_5, v_6\}$ for $U^\vee$ and $\{v_1, v_2, v_3, v_1 \wedge v_2 \wedge v_3\}$ for $M_0$. Let $m = a_0 + a_1 v_2 \wedge v_3 + a_2 v_3 \wedge v_1 + a_3 v_1 \wedge v_2$. Then the matrix of $m$ is (suppressing zero entries)
\[
\begin{bmatrix}
a_0 & \ & \ & a_1 \\
\ & a_0 & \ & a_2 \\
\ & \ & a_0 & a_3 \\
\ & a_3 & -a_2 & \ \\
-a_3 & \ & a_1 & \ \\
a_2 & -a_1 & \ & \
\end{bmatrix}
\]
which contains $3 \times 3$ minors of the form $a_i^3$ so the rank is at least $3$, and the columns $C_i$'s satisfy the relation $a_1 C_1 + a_2 C_2 + a_3 C_3 - a_0 C_4 = 0$.

\medskip
The kernel $\Ann(m)$ is an isotropic subspace with respect to $q$ because $0 = vvm = q(v,v)m$ implies $q(v,v) = 0$. From here we obtain a map from $\mathbb{P}M_0 \to P' \cup P''$ sending $\mathbb{C}m \mapsto \Ann(m)$. By symmetry, there is also a map from $\mathbb{P}M_1 \to P' \cup P''$ sending $\mathbb{C}m \mapsto \Ann(m)$.

\medskip
It follows from the continuity of the map that two $m, m'$ from the same $M_i$ will be sent to the same ruling, i.e. to the same $P'$ or $P''$.
For the inverse, and send an isotropic subspace $\mathbb{P}V' \cong \mathbb{P}^2$ in $\mathbb{P}V$ to
\[
\ker(V') = \{m \in M : V' \cdot m = 0 \text{ in } M \}.
\]
Pick a complementary isotropic subspace $\mathbb{P}V'' \subseteq \mathbb{P}V$ from the other ruling of the quadric, then $V = V' \oplus V''$. We can then identify $M_0 = (\wedge V'')_{ev} = (\wedge V')_{odd}$, $M_1 = (\wedge V'')_{odd} = (\wedge V')_{ev}$ (note that we cannot distinguish $M_0$ and $M_1$ and the codomain of this map depends on the choice of the roles of $V'$ and $V''$), and from this description we see that $\ker(V') = \wedge^3 V' \subseteq M_0$ has dimension $1$, hence we get an element in $\mathbb{P}M_0$. This gives a map $P' \cup P'' \to \mathbb{P}M_0$. It is then evident that the two maps are inverses to each others.

\medskip
Again from continuity, two $V', V''$ from the same ruling will both be sent to the same one of $\mathbb{P}M_0$ or $\mathbb{P}M_1$.
\end{proof}

Let $Y' \to W \backslash \cone(E_+ \cup E_-)$ be the restriction of the ramified double cover $\Spec B \to W$. Then $Y'$ is unramified and we have the following:

\begin{prop}
$P'$ and $P''$ can be combined into one $\mathbb{P}^3$-bundle $P_d$ over $Y'$.
\end{prop}

\begin{proof}
The two irreducible modules of $Cl_{ev}$ are classified by the action of the central element, which can be one of the two roots $y_1$, $y_2$ of $y^2 - \det(q) = y^2 - f_+ f_-$. Once we made a choice of correspondence between $P'$, $P''$ and the two modules as in Proposition \ref{prop_kuznetsov_clifford_module}, say $P' \xleftrightarrow{} y_1$ and $P'' \xleftrightarrow{} y_2$, we can define the $\mathbb{P}^3$-bundle $P_d$ for which the fiber over $(u, y_1) \in \Spec B$ is the fiber of $P'$, and the fiber over $(u, y_2)$ is the fiber of $P''$. It is well-defined as we can choose a fixed square root over an \'{e}tale neighborhood, showing that it is a $\mathbb{P}^3$-bundle in the \'{e}tale topology.
\end{proof}

Let $C'_\pm$ denote the quadric fibration $\{q_u^\pm = 0\} \subseteq \mathbb{P}V_\pm$ over the unramified double cover $Y'$.

\begin{prop}
\label{prop_p3_bundle_isomorphic}
There exist embedding of $C'_+ \times_{Y'} C'_-$ into $P_d$ such that fiber-wise it is the Segre embedding $\mathbb{P}^1 \times \mathbb{P}^1 \xhookrightarrow{} \mathbb{P}^3$.
\end{prop}

\begin{proof}
For $(u, y) \in Y'$, $p_+ \in \{q_u^+ = 0\} \subseteq \mathbb{P}V_+$ and $p_+ \in \{q_u^+ = 0\} \subseteq \mathbb{P}V_-$, the line $\ell$ joining $p_+$ and $p_-$ in $\mathbb{P}V$ is isotropic with respect to $q_u$:
\[
    q_u(ap_+ + bp_-, ap_+ + bp_-) = q_u(ap_+, ap_+) + 2 q_u(ap_+, bp_-) + q_u(bp_-, bp_-) = 0.
\]
The line $\ell$ is contained in the two isotropic subspaces of dimension $2$ with respect to $q_u$. From here we obtain a globally defined map $C'_+ \times_{Y'} C'_- \to P_d$ by sending $\ell$ to the point in the fiber of $P_d$ over $(u, y)$ which corresponds to the isotropic subspace that contains $\ell$. Now we show that \'{e}tale locally they are Segre embeddings.

\medskip
Fix an $u$. We can model the full rank quadric as $q_u = x_1 x_2 + x_3^2 - x_4^2 + x_5 x_6$ where $V_+$ has coordinates $x_1, x_2, x_3$ and $V_-$ has coordinates $x_4, x_5, x_6$. By the change of variables
\[
x_1 = z_{12},\quad x_2 = z_{34},\quad x_3 = \frac{z_{13} - z_{24}}{2},\quad x_4 = \frac{z_{13} + z_{24}}{2},\quad x_5 = z_{14},\quad x_6 = z_{23},
\]
we can identify the quadric as the image $z_{12} z_{34} - z_{13} z_{24} + z_{14} z_{23} = 0$ of the Pl\"{u}cker embedding $G(2,4) \xhookrightarrow{} \wedge^2 \mathbb{C}^4 \simeq \mathbb{C}^6$ where we use the basis $\{e_{ij} = e_i \wedge e_j : i < j\}$. For the Pl\"{u}cker embedding we can parametrize one $\mathbb{P}^3$-family of isotropic subspaces of dimension $2$ by
\[
(y_1 : y_2 : y_3 : y_4) \in \mathbb{P}^3 \to \mathbb{P} \Span \left\{ \sum_{i=1}^4 y_i e_{ij} : j = 1,\dots ,4 \right\}.
\]
Points $p_+ \in \{q_u^+ = x_1 x_2 + x_3^2 = 0\}$ can be parametrized by $(a_0 : a_1) \in \mathbb{P}^1$ as $p_+ = (a_0^2 : -a_1^2 : a_0 a_1 : 0 : 0 : 0)$. Similarly we let $p_- = (0 : 0 : 0 : a_2 a_3 : a_2^2 : a_3^2)$ for $(a_2 : a_3) \in \mathbb{P}^1$. Switching back to the coordinates $z_{ij}$'s, we have $p_+ = a_0^2 e_{12} - a_1^2 e_{34} + a_0 a_1 e_{13} - a_0 a_1 e_{24}$ and $p_- = a_2 a_3 e_{13} + a_2 a_3 e_{24} + a_2^2 e_{14} + a_3^2 e_{23}$. They come from the following matrices in $G(2,4)$:
\[
\begin{pmatrix}
a_0 & 0 & 0 & a_1 \\
0 & a_0 & a_1 & 0
\end{pmatrix}
\quad \text{and} \quad
\begin{pmatrix}
a_2 & a_3 & 0 & 0 \\
0 & 0 & a_3 & a_2
\end{pmatrix}.
\]
We are then looking for $(y_1 : y_2 : y_3 : y_4)$ whose associated isotropic subspace contains both $p_+$ and $p_-$. In $\mathbb{C}^4$, $p_+$ and $p_-$ intersect at the line spanned by $(a_0 a_2 : a_0 a_3 : a_1 a_3 : a_1 a_2)$, so they are contained in a $3$-dimensional subspace, which corresponds to an isotropic $2$-dimensional subspace in $\{q_u = 0\} \subseteq \mathbb{P}V$. It can be checked that the subspace corresponding to $(y_1 : y_2 : y_3 : y_4) = (a_0 a_2 : a_0 a_3 : a_1 a_3 : a_1 a_2)$ contains $p_+$ and $p_-$. This formula clearly is the Segre embedding. 
\end{proof}

The same is true after pulling $C'_+ \times_{Y'} C'_-$ and $P_d$ back to $\Spec A \backslash \cone(E_+ \cup E_-)$. Since the constructions in Proposition \ref{prop_p3_bundle} and \ref{prop_p3_bundle_isomorphic} are both canonical, we see that the restrictions of $P$ in Proposition \ref{prop_p3_bundle} to $\Spec A \backslash \cone(E_+ \cup E_-)$ and the pullback $P_d^*$ can be identified. This established the relation between Kuznetsov's construction and our construction outside of the ramification locus $\cone(E_+ \cup E_-)$.


\section{Generalizations and Future Work}

We can ask whether some conditions on the root stacks and the corresponding Brauer class ensure certain properties on the Enriques surfaces, and vice versa:
\begin{enumerate}[(i)]
    \item If we look at the blowup of $Y$ instead of its stacky desingularization, then the analog of the quotient $\widehat{\mathscr Y}$ appears to be the root stack of
    the rational elliptic surface at two smooth fibers. How is it related to the construction of Enriques surfaces via logarithmic transformation along these fibers, in particular how does the Brauer class come into play?
    \item When is the Enriques surface nodal, i.e. contains a nonsingular rational curve?
    \item How do the elements of the Picard group of the Enriques surface corresponds to (complexes of) sheaves on the root stack? What are the corresponding auto-equivalences?
    \item In our setup, the 2 cubic curves come from determinants of $3 \times 3$ matrices. If two quadratic forms give rise to the same cubic curves, i.e. if the root stacks are identical, is there any relation between the corresponding Enriques surfaces?
    \item What can be said about the Enriques surface if the Brauer class is trivial?
    \item Is the orbifold cohomology of the two sides related? There are no twisted sectors in the Enriques surface, but the root stack should have the cubic curves and their intersection points as twisted sectors.
\end{enumerate}

In another direction, direct generalizations of the tensor product construction of the Clifford algebras of invariant subspaces to other dimensions can be studied. For example, one can look at the  the double mirrors of a dimension $3$ intersection of $4$ quadrics in $\mathbb{P}^7$ quotient by free actions of finite groups \cite{Hua}.


\begin{thebibliography}{9}

\bibitem{orlov}
Bondal, A., Orlov, D. \textit{Reconstruction of a variety from the derived category and groups of autoequivalences}. Compositio Mathematica \textbf{125},
p.327-344 (2001).


\bibitem{caldararu}
Borisov, L., C\v{a}ld\v{a}raru, A. \textit{The Pfaffian-Grassmannian derived equivalence}. Journal of Algebraic Geometry \textbf{18}, p.201-222 (2006).


\bibitem{borisov}
Borisov, L., Li, Z. \textit{On Clifford Double Mirrors of Toric Complete Intersections}. Advances in Mathematics \textbf{328}, p.300-355 (2018).


\bibitem{bourbaki}
Bourbaki, N. \textit{Algebra I, Chapters 1-3}. Springer-Verlag Berlin Heidelberg(1989).

\bibitem{cadman}
Cadman, C. \textit{Using Stacks to Impose Tangency Conditions on Curves}. American Journal of Mathematics \textbf{129}, no. 2, p.405-427(2007).

\bibitem{caldararu2}
C\v{a}ld\v{a}raru, A. \textit{Derived Categories of Twisted Sheaves on Calabi-Yau Manifolds}. Ph.D. thesis, Cornell University (2000), also available at \url{https://people.math.wisc.edu/~andreic/publications/ThesisSingleSpaced.pdf} (last accessed July 2021).

\bibitem{cossec}
Cossec, F. \textit{Projective models of Enriques surfaces}. Math. Ann. \textbf{265},
p.283-334 (1983).

\bibitem{deligne}
Deligne, P. \textit{Notes on spinors}, in Quantum Fields and Strings: A Course for
Mathematicians. Available at \url{http://publications.ias.edu/sites/default/files/79_NotesOnSpinors.pdf}. 

\bibitem{deligne2}
Deligne, P., Morgan J. \textit{Notes on supersymmetry}, in Quantum Fields and Strings: A Course for
Mathematicians.

\bibitem{gastineau}
Gastineau-Hills, H. \textit{Quasi Clifford algebras and systems of orthogonal designs}. Journal of the Australian Mathematical Society. Series A. Pure Mathematics and Statistics, \textbf{32}(1), 
p.1-23 (1982).

\bibitem{gulbrandsen}
Gulbrandsen, M. \textit{Stack structures on GIT quotients parametrizing hypersurfaces}. Math. Nachr. 284, no. 7,
p.885-898 (2011).

\bibitem{Hua}
Hua, Z.  \textit{Classification of free actions on complete intersections of four quadrics}.
Adv. Theor. Math. Phys. 15(4):
973-990 (2011).

\bibitem{huybrechts}
Huybrechts, D. \textit{Lectures on K3 surfaces}, 1st ed. Cambridge University Press, 2016.

\bibitem{kollar}
Koll\'ar, J. \textit{Severi-Brauer varieties a geometric treatment}. Available at \url{https://arxiv.org/abs/1606.04368} (last update on 14 Jun 2016).

\bibitem{kuznetsov}
Kuznetsov, A. \textit{Derived Categories of Quadratic Fibrations and Intersections of Quadrics}. Advances in Mathematics \textbf{218},
p.1340-1369 (2008).

\bibitem{kuznetsov2}
Kuznetsov, A. \textit{Homological projective duality}. Publ. Math. IHES \textbf{105},
p.157-220 (2007).

\bibitem{kp}
Kuznetsov, A., Perry A., \textit{Categorical joins}. 
J. Amer. Math. Soc. \textbf{34} (2021), no. 2,
505-564.

\bibitem{marchuk}
Marchuk, N.G. \textit{Classification of Extended Clifford Algebras}. Izv. Vyssh. Uchebn. Zaved. Mat. 2018, no. 11,
27-32;
translation in
Russian Math. (Iz. VUZ) \textbf{62}, no. 11,
p.23-27 (2018).

\bibitem{polishchuk}
Polishchuk, A. \textit{Noncommutative Proj and Coherent Algebras}. Mathematical Research Letters \textbf{12}(1),
p.63-74 (2003).

\bibitem{verevkin}
Verevkin, A. B. \textit{On a Non-commutative Analogue of the Category of Coherent Sheaves on a
Projective Scheme}. AMS Transl. \textbf{151}, 
p.41-53 (1992).


\end{thebibliography}
\end{document}